\documentclass[12pt]{amsart}
\usepackage{amsfonts}
\usepackage{amsmath}
\usepackage{amssymb}
\usepackage[margin=1in]{geometry}
\usepackage{amsthm}   							
\usepackage{mathtools}		
\usepackage{tikz}				
\usepackage[english]{babel}

\usepackage[hidelinks]{hyperref}

\makeatletter
\@namedef{subjclassname@2020}{\textup{2020} Mathematics Subject Classification}
\makeatother

\newtheorem{theorem}{Theorem}

\newtheorem{remark}[theorem]{Remark}

\newtheorem{problem}[theorem]{Problem}

\numberwithin{equation}{section}

\newcommand{\dif}{\mathop{}\!\mathrm{d}}

\newcommand{\Z}{\mathbb{Z}}
\newcommand{\Q}{\mathbb{Q}}
\newcommand{\R}{\mathbb{R}}
\newcommand{\C}{\mathbb{C}}
\newcommand{\OK}{\mathcal{O}_K}

\newcommand{\mH}{\mathcal{H}}

\newcommand{\Gal}{\mathrm{Gal}}
\newcommand{\vol}{\mathrm{vol}}

\newcommand{\eps}{\varepsilon}
\newcommand{\vphi}{\varphi}

\DeclarePairedDelimiter\abs{\lvert}{\rvert}		
\DeclarePairedDelimiter\norm{\lVert}{\rVert}

\title{On a question of Gowers related to Littlewood's conjecture}

\author[F. Broucke]{Frederik Broucke}
\address{Alfr\'ed R\'enyi Institute of Mathematics\\ Re\'altanoda utca 13--15\\1053 Budapest\\ Hungary}
\email{broucke.frederik@renyi.hu}

\author{M\'at\'e Matolcsi}
\address{Alfr\'ed R\'enyi Institute of Mathematics, Re\'altanoda utca 13--15, 1053 Budapest, Hungary, and
Department of Analysis and Operations Research, Institute of Mathematics, Budapest University
of Technology and Economics, M\H{u}egyetem rkp. 3., H-1111 Budapest, Hungary}
\email{matomate@renyi.hu}

\author{Szil\'ard Gy. R\'ev\'esz}
\address{Alfr\'ed R\'enyi Institute of Mathematics\\ Re\'altanoda utca 13--15\\1053 Budapest\\ Hungary}
\email{revesz@renyi.hu}



\begin{document}

\begin{abstract}
In a blogpost in 2009, Gowers raised a possible approach to Littlewood's conjecture in  Diophantine approximation, leading to a question about the existence of sufficiently many points in the unit cube such that the ``hyperbolic distance" between any two of them is large. In this note we answer this question  by an explicit construction. This shows that this approach to prove Littlewood's conjecture cannot work, unless some further refinements are added.
\end{abstract}

\maketitle

\section{Introduction}
Littlewood's conjecture in Diophantine approximation states that for any real $\alpha$ and $\beta$
\[
	\liminf_{n\to\infty} n\norm{n\alpha}\norm{n\beta} = 0.
\]
Here $\norm{x}$ denotes the distance to the nearest integer of a real number $x$. This problem is notoriously hard and still open. Some partial progress has been made on the problem, most notably by Einsiedler, Katok, and Lindenstrauss \cite{EKL06}, who showed  that the set of exceptions $(\alpha,\beta) \in \R^2$ to Littlewood's conjecture has Hausdorff dimension $0$.

In a 2009 blogpost about problems related to Littlewood's conjecture, Gowers posed the following question \cite[Problem 3]{Go09}:
\begin{problem}
\label{Gowers' question}
Is there a constant $c>0$ such that for every positive integer $n$ it is possible to find $n$ points in $[0,1]^3$ such that the ``hyperbolic distance" between any two distinct points is at least $c/n$? Here the ``hyperbolic distance'' between points $(x_1,x_2,x_3)$ and $(y_1,y_2,y_3)$ is defined as $\abs{(x_1 - y_1)(x_2 - y_2)(x_3 - y_3)}$.
\end{problem}

The connection with Littlewood's conjecture is as follows. If there exists a counterexample $(\alpha, \beta)$ to Littlewood's conjecture, then the answer to the above question is yes, as one can take
\[
	X_n = \bigl\{x_k:=\bigl(k/n, \{k\alpha\}, \{k\beta\}\bigr) : 1\le k \le n\bigr\},
\]
where $\{y\}$ denotes the fractional part of the real number $y$. The distance between distinct points $x_k$ and $x_l$ with $k<l$ is at least $((l-k)/n) \norm{(l-k)\alpha}\norm{(l-k)\beta} \ge c/n$, in view of $\liminf_{m\to\infty} m\norm{m\alpha}\norm{m\beta} > 0$.

In this short note, we provide an unconditional affirmative answer to Gowers' question (and its analogs in any dimension $d$) by a lattice construction using some ideas from algebraic number theory.
\section{The construction}
Let us denote by $H^d$ the region bounded by the unit hyperbola in $\R^d$, i.e.\ the open set
\[
	H^d = \{x = (x_1, \dotsc, x_d) \in \R^d: \abs{x_1 \dotsm x_d} < 1\}.
\]
\begin{theorem}
\label{construction}
For any dimension $d$ there exists a full-rank lattice $\Lambda$ in $\R^d$ which intersects trivially with the hyperbolic region: $\Lambda \cap H^d = \{0\}$.
\end{theorem}
\begin{proof}
Fix\footnote{For example, fix a prime $p \equiv 1 \mod 2d$ and let $K$ be the fixed field inside the cyclotomic extension $\Q(\zeta_p)$ of the (normal) Galois subgroup $H$ generated by $\zeta_p \mapsto \zeta_p^{d}$. This field has order $\frac{p-1}{(p-1)/d} = d$ and is real as $H$ contains complex conjugation (the unique element of $\Gal(\Q(\zeta_p)/\Q)$ of order $2$).} a real Galois extension $K/\Q$ of degree $d$, and let $\OK$ be its ring of integers. Take $\Lambda$ to be the Minkowski embedding of $\OK$. Writing $\Gal(K/\Q) = \{\sigma_i: 1\le i \le d\}$, it is given by
\[
	\Lambda = \bigl\{(\sigma_1(\alpha), \dotsc, \sigma_d(\alpha)): \alpha \in \OK\bigr\} \subseteq (\OK)^d \subseteq \R^d.
\]
This is a full rank lattice; if $(\alpha_1, \dotsc, \alpha_d)$ is an integral basis for $\OK$, then $\Lambda = v_1\Z + \dotsb + v_d\Z$ with $v_j = (\sigma_1(\alpha_j), \dotsc, \sigma_d(\alpha_j))$. Now for $\lambda = (\sigma_1(\alpha), \dotsc, \sigma_d(\alpha)) \in \Lambda$ there holds
\[
	\abs{\sigma_1(\alpha) \dotsm \sigma_d(\alpha)} = \abs{N_{K/\Q}(\alpha)},
\]
where $N_{K/\Q}$ denotes the field norm of the field extension $K/\Q$. For an algebraic integer $\alpha \in \OK$, this norm is an integer, and therefore less than $1$ in absolute value if and only if it is $0$. But if $N_{K/\Q}(\alpha) = 0$ then $\alpha= 0$ and $\lambda = 0$. Hence $\Lambda \cap H^d = \{0\}$.
\end{proof}

\begin{remark}
In the construction it is not necessary that $K$ is a Galois extension; it suffices to take a totally real extension $K$ of degree $d$, i.e.\ a degree $d$ extension all of whose field embeddings in $\C$ end up in $\R$. Also one can consider the Minkowski embedding of any additive subgroup of $\OK$ of finite index, such as non-zero ideals $I \trianglelefteq \OK$ or subrings of the form $\Z[\alpha]$ where the degree of $\alpha$ is $d$. This reduces the density of the resulting lattice however.
\end{remark}

For example, in dimension two we can take $\Lambda = (1,1)\Z + (\sqrt{2}, -\sqrt{2})\Z$; in dimension three we can take $\alpha = \zeta_7 + \overline{\zeta_7} = 2\cos(2\pi/7)$, $K = \Q(\alpha)$, and
\[
	\Lambda = (1,1,1)\Z + 2\bigl(\cos(\tfrac{2\pi}{7}), \cos(\tfrac{4\pi}{7}), \cos(\tfrac{6\pi}{7})\bigr)\Z + 4\bigl(\cos^2(\tfrac{2\pi}{7}), \cos^2(\tfrac{4\pi}{7}), \cos^2(\tfrac{6\pi}{7})\bigr)\Z.
\]

The covolume of the lattice $\Lambda$ in the construction of Theorem \ref{construction} is given by $\vol(\R^d/\Lambda) = \sqrt{\Delta_K}$, where $\Delta_K$ is the discriminant of $K$ (which is positive if $K$ is totally real). In the above examples the covolume is $2\sqrt{2}$ and $7$ respectively.

\medskip

Any full-rank lattice $\Lambda$ with $\Lambda\cap H^d = \{0\}$ provides now an affirmative answer to Problem \ref{Gowers' question}. Indeed, let $R$ be large and consider the point set
\[
	Y_R = \bigl\{\lambda/R^{1/d}: \lambda \in \Lambda \cap [0,R^{1/d}]^d\bigr\}.
\]
As $\Lambda$ is full rank we have
\[
	\#Y_R \sim \frac{R}{\vol(\R^d/\Lambda)}.
\]
Moreover, if $\lambda/R^{1/d}$ and $\lambda'/R^{1/d}$ are distinct points of $Y_R$, then
\[
	\abs[\bigg]{\frac{(\lambda_1 - \lambda'_1)}{R^{1/d}} \dotsm \frac{(\lambda_d - \lambda'_d)}{R^{1/d}}} > \frac{1}{R}.
\]
So, asymptotically, any constant $c < 1/\vol(\R^d/\Lambda)$ is admissible in Gowers' question. In particular, in dimension $d$ we can asymptotically reach the constant
\begin{equation}
\label{cd}
	c_d = \Bigl(\min\bigl\{\Delta_K: K \text{ totally real number field of degree } d\bigr\}\Bigr)^{-1/2}.
\end{equation}
Determining the minimal discriminant among all number fields of a fixed signature is an important and well-studied problem. It is well known that the minimal discriminant of a totally real quadratic extension is $5$ (attained by $K = \Q(\sqrt{5})$, $\OK = \Z[\frac{1+\sqrt{5}}{2}]$) and of a totally real cubic extension is $49$ (attained by the above example); see e.g.\ \cite[Table 1]{Od90} for the minimal values up to $d=8$. So $c_2 = 1/\sqrt{5}$ and $c_3 = 1/7$.

\begin{remark}
It is known \cite[Theorem 7]{CSD55} that if $\alpha$ and $\beta$ belong to the same cubic field extension of $\Q$, then Littlewood's conjecture holds for $\alpha$ and $\beta$. Curiously, the sets $Y_R$ we constructed exhibiting a ``counterexample to the generalized Littlewood conjecture'' (in the sense of Problem \ref{Gowers' question}) all consist (up to a scaling factor) of elements lying in some fixed cubic field extension of $\Q$.
\end{remark}

\begin{remark}
One can wonder whether there are other constructions for Problem \ref{Gowers' question} which yield a better constant $c$ than the above lattice construction. In dimension $2$ there is another construction, namely
\[
	X_n = \bigl\{\bigl(k/n, \{k\alpha\}\bigr): 1\le k \le n\bigr\}
\]
where $\alpha$ is a badly approximable number (a ``counterexample to the $1$-dimensional Littlewood conjecture"). The asymptotic best value of $c$ one can reach here is $1/\sqrt{5}$ by taking $\alpha = \frac{1+\sqrt{5}}{2}$, which matches $c_2$ (although this point set is quite different from the lattice construction).
\end{remark}

\section{A connection with the Delsarte problem}
Let $\mathbb{T}^d$ be the $d$-dimensional torus which we identify with $[-1/2, 1/2)^d$, and for a given $0 < \eps \le 2^{-d}$ let
\[
	\mathcal{H}^d_{\eps} = \bigl\{x = (x_1, \dotsc, x_d) \in \mathbb{T}^d \cong [-1/2, 1/2)^d: \abs{x_1 \dotsm x_d} < \eps\bigr\}.
\]
Finally, let $N_d(\eps)$ be the maximal size of a subset $Z \subseteq \mathbb{T}^d$ with the property that any two distinct elements $x$ and $y$ from $Z$ satisfy $x-y\not\in \mathcal{H}_{\eps}^d$. Problem \ref{Gowers' question} is closely related to asking whether $N_{d}(\eps) \gg 1/\eps$ when $\eps \to 0$, and the lattice construction from Theorem \ref{construction} shows that this indeed holds (in fact in any dimension). Namely, take
\[
	Z_{\eps} = \bigl\{\eps^{1/d}\lambda: \lambda \in \Lambda \cap [-\eps^{-1/d}/4,\eps^{-1/d}/4]^d\bigr\},
\]
which has size $\# Z_{\eps} \sim \bigl(2^{d}\vol(\R^d/\Lambda) \eps\bigr)^{-1}$.
(The additional scaling by the factor $1/2$ ensures that the difference $Z_{\eps} -Z_{\eps}$ in $\mathbb{T}^d$ is the same as when taken in $\R^d$.)

With this slight reformulation, there is a natural connection with the Delsarte problem, which is a classical extremal problem in Fourier analysis. Let $\mathcal{F}(\mH^d_{\eps})$ be the class of real-valued continuous functions $\vphi$ on $\mathbb{T}^d$ which are non-positive on $\mathbb{T}^d \setminus \mH^{d}_{\eps}$, are normalized to $\vphi(0)=1$, and are positive definite, meaning that all the Fourier coefficients of $\vphi$ are non-negative.
The Delsarte problem for $\mathcal{H}_{\eps}^d$ asks to determine the constant
\[
	\mathcal{D}(\mathcal{H}_{\eps}^d) = \sup\biggl\{\int_{\mathbb{T}^d} \vphi: \vphi \in \mathcal{F}(\mH^d_{\eps})\biggr\}.
\]
The quantity $\mathcal{D}(\mathcal{H}_{\eps}^d)$ is called the {\it Delsarte constant} of $\mathcal{H}_{\eps}^d$.

\medskip

Next, let $\mathcal{M}(\mH^d_{\eps})$ be the class of positive measures $\nu$ on $\mathbb{T}^d$ which are supported on $\{0\} \cup (\mathbb{T}^d \setminus \mH^d_{\eps})$, are normalized to $\nu(\{0\})=1$, and are positive definite, and
\[
	\mathcal{D}'(\mH^d_{\eps}) = \sup\bigl\{\nu(\mathbb{T}^d): \nu \in \mathcal{M}(\mH^d_{\eps})\bigr\}.
\]
It is known \cite{BFGRR} that $\mathcal{D}(\mH_{\eps}^d)\mathcal{D}'(\mH_{\eps}^d) = 1$, a statement which is called ``strong duality" for the Delsarte problem\footnote{In standard linear programming strong duality means that the primal and dual problems have equal value. This notation is used in \cite{BFGRR}. Here we adopt a multiplicative version of formulation, which is obtained by renormalizing the dual problem.}.

If $Z$ is a set of points with the above separation property, then the measure
\[
	\nu = \frac{1}{\#Z}\delta_Z \ast \delta_{-Z}, \quad \delta_Z = \sum_{x\in Z} \delta_x,
\]
belongs to the class $\mathcal{M}(\mH^d_{\eps})$, and its total mass is $\nu(\mathbb{T}^d) = \#Z$. Hence it follows that $N_d(\eps) \le \mathcal{D}'(\mH^d_{\eps})$.

On the other hand, we may consider
\[
	\vphi = \frac{1}{\eps}\chi \ast \chi, \quad \chi \text{ the indicator function of } [-\eps^{1/d}/2, \eps^{1/d}/2]^d,
\]
which clearly belongs to $\mathcal{F}(\mH^d_{\eps})$ and has total integral $\int_{\mathbb{T}^d} \vphi = \eps$. Hence $\mathcal{D}(\mH_{\eps}^d) \ge \eps$ and so $\mathcal{D}'(\mH_{\eps}^d) \le 1/\eps$. The latter can also be seen as follows without appealing to the strong duality statement: for any $\nu \in \mathcal{M}(\mH_{\eps}^d)$, also the measure $\nu - \nu(\mathbb{T}^d)\dif x$ is positive definite (here $\dif x$ denotes the normalized Haar measure on $\mathbb{T}^d$.) Hence
\[
	0 \le \langle \nu - \nu(\mathbb{T}^d)\dif x, \vphi \rangle = 1 - \nu(\mathbb{T}^d)\int_{\mathbb{T}^d}\vphi = 1 - \eps \nu(\mathbb{T}^d).
\]

In particular, $N_d(\eps) \le 1/\eps$, so in fact $N_d(\eps) \asymp 1/\eps$. Recalling \eqref{cd} we have for the Delsarte constants the inequalitites
\[
	\eps \le \mathcal{D}(\mH_{\eps}^d) \le \frac{2^d \eps}{c_d}(1+o(1)), \quad \frac{c_d}{2^d \eps}(1+o(1)) \le \mathcal{D}'(\mH_{\eps}^d) \le \frac{1}{\eps}, \quad \eps \to 0.
\]

\section{Acknowledgments}
The authors are grateful to Imre Ruzsa for discussing the connection of the problem to the Delsarte constant of the hyperbola. The research was carried out while Frederik Broucke was on a R\'enyi Postdoctoral Fellowship at the Alfr\'ed R\'enyi Institute of Mathematics.
M\'at\'e Matolcsi was supported by grants NKFIH-154121 and NKFIH-146387. Szil\'ard Gy. R\'ev\'esz was supported in part by Hungarian National Research, Development and
Innovation Office, Grant \#'s K-146387, K-147153 and Excellence No.151341. The authors have not used AI tools for the results of this paper.

\end{document}